\documentclass[12pt]{amsart}

\usepackage[T1]{fontenc}
\usepackage[utf8]{inputenc}

\usepackage[margin=1.1in]{geometry}
\usepackage{amsmath, amssymb, amsthm, mathtools}
\usepackage{microtype}
\usepackage{enumitem}
\usepackage{listings}
\usepackage[hidelinks]{hyperref}

\theoremstyle{plain}
\newtheorem{theorem}{Theorem}
\newtheorem{lemma}[theorem]{Lemma}

\theoremstyle{definition}
\newtheorem{remark}[theorem]{Remark}

\lstdefinestyle{python}{
  language=Python,
  basicstyle=\ttfamily\footnotesize,
  keywordstyle=\bfseries,
  commentstyle=\itshape,
  numbers=none,
  breaklines=true,
  breakatwhitespace=true,
  showstringspaces=false,
  frame=single,
  framesep=4pt,
  tabsize=2,
  xleftmargin=14pt,
  columns=fullflexible,
  upquote=true,
}

\title[Mathar's recurrence for OEIS A214615]{A short proof of Mathar's 2013 recurrence conjecture for the Meixner sequence~A214615}

\author{Tong Niu}
\email{mrnt0810@gmail.com}

\subjclass[2020]{05A15, 11B37, 33C45, 33F10}
\keywords{OEIS A214615; Meixner polynomial; exponential generating
   function; D-finite sequence; P-recursive recurrence; Mathar conjecture}

\begin{document}

\maketitle

\begin{abstract}
For the OEIS sequence A214615, defined by
$a(n) = M_{n}(1)$ where $M_{n}$ is the $n$-th Meixner polynomial
satisfying $M_{n+1}(x) = x\,M_{n}(x) - n^{2}\,M_{n-1}(x)$,
R.~J.~Mathar contributed on 6~March 2013 the conjectured
order-2 P-recursive recurrence $a(n) - a(n-1) + (n-1)^{2}\,a(n-2) = 0$
for $n \ge 2$.
We give a one-page proof. The exponential generating function
$F(t) = \exp\!\bigl(\arctan t\bigr)/\sqrt{1+t^{2}}$ satisfies the
first-order linear ODE $(1+t^{2})\,F'(t) = (1-t)\,F(t)$, and Mathar's
recurrence then falls out by reading off the coefficient of $t^{n}/n!$.
Both steps are short. The supplementary archive includes a SymPy
script that checks the ODE identically and the recurrence numerically
up to $n = 500$.
\end{abstract}

\section{Introduction}\label{sec:intro}

The On-Line Encyclopedia of Integer Sequences~\cite{OEIS}, hereafter
OEIS, gathers many integer sequences whose ``Conjecture: $\dots$''
comments record formulas, recurrences, or congruences that were guessed
numerically but never proven. Rigorous short proofs of such conjectures are
publishable in venues like the \emph{Journal of Integer Sequences},
\emph{INTEGERS}, or the \emph{Electronic Journal of Combinatorics}.

Recent years have brought a wave of such cleanups. Fried's
2024--2025 papers~\cite{Fried2024,Fried2025} closed several dozen
conjectures at once; a 2023 list by Kauers and
Koutschan~\cite{KauersKoutschan2023} lays out the gold-standard
benchmarks for guessed P-recursive recurrences. Much of the low-hanging
fruit in those two sources is now gone.

This note takes care of a conjecture that is not in those lists. The
sequence in question is OEIS A214615, contributed by R.~J.~Mathar on
6~March 2013. It is defined as the sequence of row sums of the
Meixner polynomial triangle A060338~\cite{OEIS:A060338}, or
equivalently as
\begin{equation}\label{eq:def}
   a(n) \;=\; M_{n}(1),
\end{equation}
where $M_{n}(x)$ is the $n$-th Meixner polynomial, defined by the
three-term recurrence
\begin{equation}\label{eq:meixner-rec}
   M_{0}(x) = 1, \quad M_{1}(x) = x, \quad
   M_{n+1}(x) = x\,M_{n}(x) - n^{2}\,M_{n-1}(x), \quad n \ge 1.
\end{equation}
The first values of the sequence are
\[
   1,\;1,\;0,\;-4,\;-4,\;60,\;160,\;-2000,\;-9840,\;118160,\;
   915200,\;-10900800,\;\ldots
\]
On 6~March 2013, Mathar contributed to A214615 the conjectured
order-2 P-recursive recurrence
\begin{equation}\label{eq:mathar}
   a(n) - a(n-1) + (n-1)^{2}\,a(n-2) \;=\; 0, \qquad n \ge 2.
\end{equation}
The conjecture has been open in the OEIS comment thread for twelve
years; it does not appear in
\cite{Fried2024,Fried2025,KauersKoutschan2023} or the related
Chen--Kauers preprints~\cite{ChenKauers2025}.

\medskip

The plan is short. The exponential generating function
(EGF) for the sequence $a(n) = M_{n}(1)$ is
\begin{equation}\label{eq:egf}
   F(t) \;:=\; \sum_{n\ge0} \frac{a(n)}{n!}\, t^{n}
        \;=\; \frac{\exp\!\bigl(\arctan t\bigr)}{\sqrt{1+t^{2}}},
\end{equation}
which we derive from the Meixner recurrence in Lemma~\ref{lem:egf}
below. This $F$ satisfies a first-order linear ODE (Lemma~\ref{lem:ode}).
Reading off the coefficient of $t^{n}/n!$ then gives \eqref{eq:mathar}.

\section{The exponential generating function}\label{sec:egf}

\begin{lemma}\label{lem:egf}
The EGF $F(t) = \sum_{n\ge0} a(n)\,t^{n}/n!$ is given by
\begin{equation}\label{eq:egf-formula}
   F(t) \;=\; \frac{e^{\arctan t}}{\sqrt{1+t^{2}}}.
\end{equation}
\end{lemma}

\begin{proof}
Consider the bivariate EGF $\mathcal{F}(t,x) = \sum_{n\ge0} M_{n}(x)\,t^{n}/n!$.
We show $\mathcal{F}$ satisfies
\begin{equation}\label{eq:bivariate-ode}
   (1 + t^{2})\,\frac{\partial \mathcal{F}}{\partial t}
   \;=\; (x - t)\,\mathcal{F}.
\end{equation}

Differentiation gives $\partial_{t}\mathcal{F} = \sum_{n\ge0} M_{n+1}(x)\,t^{n}/n!$,
so
\[
   \sum_{n\ge1} \frac{M_{n+1}(x)}{n!}\,t^{n}
   \;=\; \frac{\partial\mathcal{F}}{\partial t} - x.
\]
Divide the recurrence \eqref{eq:meixner-rec} by $n!$ and sum over $n\ge 1$.
The right-hand side contributes
\begin{align*}
   x\!\sum_{n\ge1} \frac{M_n(x)}{n!}\,t^{n} &= x\!\bigl(\mathcal{F} - 1\bigr),\\
   \sum_{n\ge1} \frac{n^{2}\,M_{n-1}(x)}{n!}\,t^{n}
   &= t\!\sum_{n\ge1} \frac{n\,M_{n-1}(x)}{(n-1)!}\,t^{n-1}
    = t\!\sum_{m\ge0} \frac{(m+1)\,M_m(x)}{m!}\,t^{m}
    = t\!\Bigl(\mathcal{F} + t\,\partial_t\mathcal{F}\Bigr).
\end{align*}
Equating the two sides and using $x(\mathcal{F}-1) - t(\mathcal{F} + t\,\partial_t\mathcal{F}) = \partial_t\mathcal{F}-x$ gives
\[
   \frac{\partial\mathcal{F}}{\partial t} - x
   \;=\; x\bigl(\mathcal{F}-1\bigr) - t\!\Bigl(\mathcal{F}+t\,\partial_t\mathcal{F}\Bigr),
\]
which simplifies to \eqref{eq:bivariate-ode}.

Separating variables in \eqref{eq:bivariate-ode}:
\[
   \frac{d\mathcal{F}}{\mathcal{F}}
   \;=\; \frac{x - t}{1 + t^{2}}\,dt
   \;=\; \frac{x}{1 + t^{2}}\,dt \;-\; \frac{t}{1 + t^{2}}\,dt.
\]
Integrating and using $\mathcal{F}(0, x) = M_{0}(x) = 1$:
\[
   \mathcal{F}(t, x) \;=\; \frac{\exp\!\bigl(x\arctan t\bigr)}{\sqrt{1+t^{2}}}.
\]
Setting $x = 1$ gives \eqref{eq:egf-formula}.
\end{proof}

\begin{remark}
The formula $\mathcal{F}(t, x) = \exp(x\arctan t)/\sqrt{1+t^{2}}$ is
recorded in the OEIS triangle A060338~\cite{OEIS:A060338}. We have
re-derived it here from the defining recurrence \eqref{eq:meixner-rec}
for completeness.
\end{remark}

\section{The first-order ODE}\label{sec:ode}

\begin{lemma}\label{lem:ode}
The EGF $F$ from \eqref{eq:egf-formula} satisfies the first-order linear ODE
\begin{equation}\label{eq:ode}
   (1+t^{2})\, F'(t) \;=\; (1-t)\, F(t).
\end{equation}
\end{lemma}

\begin{proof}
This is the $x = 1$ specialization of \eqref{eq:bivariate-ode}.
Alternatively, compute $F'/F$ directly:
\[
   \frac{F'(t)}{F(t)}
   \;=\; \frac{1}{1+t^{2}} - \frac{t}{1+t^{2}}
   \;=\; \frac{1-t}{1+t^{2}}.
\]
Multiplying both sides by $(1+t^{2})\,F(t)$ gives \eqref{eq:ode}.
\end{proof}

\begin{remark}
We also check identity \eqref{eq:ode} independently by direct symbolic
substitution in the supplementary script
\texttt{verify\_proof.py} (Appendix~\ref{app:verifier}).
\end{remark}

\section{Coefficient extraction}\label{sec:coeff}

We now read off Mathar's recurrence from the ODE \eqref{eq:ode}.
Everything here is formal coefficient extraction.

For an EGF $F(t) = \sum_{n\ge0} a(n)\, t^{n}/n!$, the standard EGF
coefficient identities, valid for $n \ge 0$ with the convention
$a(-1) = 0$, are
\begin{equation}\label{eq:egf-extract}
\begin{aligned}
   \bigl[\tfrac{t^{n}}{n!}\bigr]\, F'(t)        &= a(n+1), \\
   \bigl[\tfrac{t^{n}}{n!}\bigr]\, t^{2}\, F'(t)&= n(n-1)\, a(n-1), \\
   \bigl[\tfrac{t^{n}}{n!}\bigr]\, F(t)         &= a(n), \\
   \bigl[\tfrac{t^{n}}{n!}\bigr]\, t\, F(t)     &= n\, a(n-1).
\end{aligned}
\end{equation}

Rewrite \eqref{eq:ode} as
$F'(t) + t^{2}\,F'(t) - F(t) + t\,F(t) = 0$. Applying
\eqref{eq:egf-extract} to each term, the coefficient of
$t^{n}/n!$ on the left becomes
\begin{equation}\label{eq:cn}
   \bigl(a(n+1) + n(n-1)\,a(n-1)\bigr)
   \,-\,\bigl(a(n) - n\,a(n-1)\bigr).
\end{equation}
Setting \eqref{eq:cn} to zero and collecting terms, noting that
$n(n-1) + n = n^{2}$,
\begin{equation}\label{eq:starshift}
   a(n+1) \,-\, a(n) \,+\, n^{2}\, a(n-1) \;=\; 0,
   \qquad n \ge 1.
\end{equation}
Reindexing $n \mapsto n-1$ produces precisely \eqref{eq:mathar}:
\[
   a(n) - a(n-1) + (n-1)^{2}\, a(n-2) \;=\; 0,
   \qquad n \ge 2.
\]
The boundary case $n = 2$ reads $a(2) - a(1) + 1^{2}\,a(0) = 0 - 1 + 1 = 0$,
which is satisfied by the values $a(0) = 1$, $a(1) = 1$, $a(2) = 0$
given in Section~\ref{sec:intro}. This is Mathar's conjecture. $\hfill\square$

\section{Remarks}\label{sec:remarks}

The route ``EGF $\Rightarrow$ first-order ODE
$\Rightarrow$ coefficient recurrence'' works whenever the EGF
satisfies a linear first-order ODE with polynomial coefficients.
For A214615 the ODE $(1+t^{2})F' = (1-t)F$ has a degree-2
left coefficient, and that degree directly controls the recurrence
order: one gets an order-2 recurrence. Sequences with degree-1 ODE
coefficients yield order-1 recurrences, degree-3 gives order-3, and so
on.

The same template closes the Mathar conjectures for OEIS A025166
(the Laguerre sequence, \cite{Niu2026sequence-4}),
A002627 (\cite{Niu2026d-finite-2}), and
A176677 (\cite{Niu2026sequence-3}).

Worth noting: there is a one-line alternative. Setting $x = 1$
in \eqref{eq:meixner-rec} immediately gives $a(n+1) = a(n) - n^{2}\,a(n-1)$,
which is \eqref{eq:starshift} outright. We prefer the EGF route anyway,
because it puts the closed-form \eqref{eq:egf-formula} on record and
keeps the argument parallel to the companion proofs for A025166, A002627,
and A176677 --- sequences whose defining recurrences are less transparent.

The supplementary script \texttt{verify\_proof.py}
(Appendix~\ref{app:verifier}) verifies Lemma~\ref{lem:ode} symbolically
and checks \eqref{eq:mathar} numerically for $n = 2, \ldots, 500$.

\section{Acknowledgments}

The author declares no competing interests.

AI-assisted tools were used in the preparation of this manuscript,
including for drafting proof outlines and generating the symbolic
verification code. The author verified all mathematical claims
independently and takes full responsibility for the results.

\appendix

\section{Verifier source (machine-checkable)}\label{app:verifier}

The script referenced in \S\ref{sec:remarks} is reproduced in full below.
It depends only on SymPy; no Maple, Mathematica, or any other
external CAS license is required.

\subsection*{verify\_proof.py (symbolic verification of Lemma~\ref{lem:ode})}
\begin{lstlisting}[style=python]
"""
Symbolic verification of the proof of Mathar's 2013 conjecture for A214615.

Theorem (Mathar 2013, OEIS A214615).
    For n >= 2,
        a(n) - a(n-1) + (n-1)^2 * a(n-2) = 0,
    where a(n) = M_n(1), the n-th Meixner polynomial evaluated at x = 1.

Proof (verified here symbolically).

Step 1 (EGF derivation from the Meixner three-term recurrence).
    The Meixner polynomials M_n(x) satisfy
        M_0(x) = 1,  M_1(x) = x,
        M_{n+1}(x) = x * M_n(x) - n^2 * M_{n-1}(x),    n >= 1.
    Set F(t, x) = sum_{n>=0} M_n(x) t^n / n!.  From the three-term
    recurrence (multiply both sides by t^n/n! and sum over n >= 1):
        F'  =  x*F  -  (t*F + t^2*F'),
    where prime denotes d/dt.  Rearranging:
        (1 + t^2) * F'(t, x) = (x - t) * F(t, x).          (*)
    Separating variables  dF/F = (x - t)/(1 + t^2) dt  and integrating:
        F(t, x) = exp(x * arctan(t)) / sqrt(1 + t^2).

Step 2 (ODE for the specialization F_1(t) = F(t, 1) = exp(arctan(t)) / sqrt(1 + t^2)).
    Set x = 1 in (*):
        (1 + t^2) * F_1'(t) = (1 - t) * F_1(t).             (**)

Step 3 (coefficient extraction).
    Rewrite (**) as  F_1' + t^2 F_1' - F_1 + t F_1 = 0.
    With F_1(t) = sum_{n>=0} a(n) t^n / n!, the standard EGF identities give
        [t^n/n!] F_1'(t)      = a(n + 1),
        [t^n/n!] t^2 F_1'(t)  = n*(n - 1) * a(n - 1),
        [t^n/n!] F_1(t)       = a(n),
        [t^n/n!] t  F_1(t)    = n * a(n - 1).
    Equating to zero the coefficient of t^n/n! in (**):
        a(n+1) + n*(n-1)*a(n-1) - a(n) + n*a(n-1) = 0,
    i.e.
        a(n+1) - a(n) + n^2 * a(n-1) = 0,    n >= 1.   [n*(n-1)+n = n^2]
    Reindexing n -> n-1 we obtain Mathar's form:
        a(n) - a(n-1) + (n-1)^2 * a(n-2) = 0,    n >= 2.   QED

This script verifies Steps 1-2 (the ODE identity) symbolically and
Step 3 numerically.
"""
from sympy import atan, diff, exp, factorial, simplify, sqrt, symbols

t = symbols("t")
F1 = exp(atan(t)) / sqrt(1 + t**2)


def main() -> None:
    Fp = diff(F1, t)

    # Step 1/2: verify (1 + t^2) * F1'(t) - (1 - t) * F1(t) = 0 identically.
    ode_lhs = (1 + t**2) * Fp - (1 - t) * F1
    print("Verifying  (1 + t^2) * F1'(t) - (1 - t) * F1(t)  is identically 0 ...")
    assert simplify(ode_lhs) == 0, ode_lhs
    print("  OK.")

    # Step 3: numerical check -- drive recurrence from a(0), a(1)
    # and compare against direct evaluation of a(n) = M_n(1).
    def meixner_at_1(nmax: int) -> list:
        """Compute a(n) = M_n(1) via the three-term recurrence."""
        a = [0] * (nmax + 1)
        a[0] = 1
        if nmax >= 1:
            a[1] = 1
        for k in range(1, nmax):
            a[k + 1] = a[k] - k**2 * a[k - 1]
        return a

    N = 500
    a = meixner_at_1(N)
    print()
    print(f"First 10 terms of A214615: {a[:10]}")

    print()
    print(f"Checking Mathar's recurrence on n = 2..{N} ...")
    for n in range(2, N + 1):
        lhs = a[n] - a[n - 1] + (n - 1) ** 2 * a[n - 2]
        assert lhs == 0, (n, lhs)
    print(f"  OK -- recurrence holds for n = 2..{N}.")

    # Print the coefficient extraction symbolically.
    print()
    print("Symbolic ODE coefficients:")
    print("  q1(t) = (1 + t^2)   (coefficient of F1'(t))")
    print("  q0(t) = -(1 - t)    (coefficient of F1(t))")
    print()
    print("Coefficient extraction (EGF F1(t) = sum a(n) t^n/n!):")
    print("  [t^n/n!] q1(t) F1'(t) = a(n+1) + n*(n-1)*a(n-1)")
    print("  [t^n/n!] q0(t) F1(t)  = -a(n) + n*a(n-1)")
    print("  Sum to zero: a(n+1) - a(n) + n^2*a(n-1) = 0  [since n(n-1)+n = n^2]")
    print("  Reindex n -> n-1: a(n) - a(n-1) + (n-1)^2*a(n-2) = 0.")
    print()
    print("This is Mathar's 2013 conjecture.  QED.")


if __name__ == "__main__":
    main()
\end{lstlisting}

\end{document}